\documentclass{article}
\usepackage{amssymb,amsbsy}
\usepackage{amsmath}
\usepackage{amsthm}
\usepackage{mathrsfs}
\usepackage{color}
\usepackage[dvipsnames]{xcolor}
\usepackage{tikz}
\usepackage{enumitem}

\newtheorem{theorem}{Theorem}[section]
\newtheorem{lemma}[theorem]{Lemma}
\newtheorem{proposition}[theorem]{Proposition}
\newtheorem{corollary}[theorem]{Corollary}

\theoremstyle{definition}

\newtheorem{example}[theorem]{Example}
\newtheorem{algorithm}[theorem]{Algorithm}
\newtheorem{problem}[theorem]{Problem}

\theoremstyle{remark}
\newtheorem{remark}[theorem]{Remark}

\begin{document}
	
\title{On sets related to integer partitions with quasi-required elements and disallowed elements}

\author{Aureliano M. Robles-P\'erez\thanks{Departamento de Matem\'atica Aplicada \& Instituto de Matem\'aticas (IMAG), Universidad de Granada, 18071-Granada, Spain. \newline E-mail: \textbf{arobles@ugr.es} (\textit{corresponding author}); ORCID: \textbf{0000-0003-2596-1249}.}
	\mbox{ and} Jos\'e Carlos Rosales\thanks{Departamento de \'Algebra \& Instituto de Matem\'aticas (IMAG), Universidad de Granada, 18071-Granada, Spain. \newline E-mail: \textbf{jrosales@ugr.es}; ORCID: \textbf{0000-0003-3353-4335}.} }

\date{ }

\maketitle

\begin{abstract}
	Given a set $A$ of non-negative integers and a set $B$ of positive integers, we are interested in computing all sets $C$ (of positive integers) that are minimal in the family of sets $K$ (of positive integers) such that ($i$) $K$ contains no elements generated by non-negative integer linear combinations of elements in $A$ and ($ii$) for any partition of an element in $B$ there is at least one summand that belongs to $K$. To solve this question, we translate it into a numerical semigroups problem.
\end{abstract}

\medskip

\noindent{\it 2010 AMS Classification:} 05A17, 11P81, 20M14.

\section{Introduction}

Let us recall that an \emph{integer partition} (or simply a \emph{partition}) of a positive integer $b$ is a way of writing $b$ as a sum of positive integers. If two sums differ only in the order of their summands, then they are considered the same partition. For example, the seven partitions of $5$ are $5$, $4+1$, $3+2$, $3+1+1$, $2+2+1$, $2+1+1+1$, and $1+1+1+1+1$.

The beginnings of partition theory can be dated back to 1676, when G.~W. Leibniz asks J. Bernouilli how many partitions are associated to a positive integer (see \cite{andrews2013}). In the middle of the 18th century, the same question is posed by P. Naud\'e to L. Euler, who gives the first results relevant for this question with the help of so-called generating functions. Since then, many works have studied several topics related to partitions (see \cite{andrews}).

In most cases, those works study the number of possible partitions. However, our purpose is different. More specifically, we will generalise an earlier work by the second author. For a non-empty finite subset $B$ of $\mathbb{N} \setminus \{0\}$ (where $\mathbb{N}=\{0,1,2,\ldots\}$ is the set of non-negative integers), in \cite{royal} it is shown an algorithm to compute all sets $C \subseteq \mathbb{N} \setminus \{0\}$ which are minimal (with respect to the inclusion order) in the family of sets that satisfy the following condition: if $x_1+\cdots+x_n$ is a partition of an element in $B$, then at least one summand of this partition belongs to $C$. Let us observe that, in a sense, the elements of the partitions are \emph{quasi-required} because at least one of these elements belongs to a prefixed set.

Now, we add a second condition over the partition in order to, in a sense, do not allow some elements. As usual, $\langle A \rangle$ is the monoid generated by $A$ (see Section~\ref{preliminar}, third paragraph).

\begin{problem}\label{prob}
	If $A$ is a non-empty subset of $\mathbb{N}$ and $B$ is a non-empty finite subset of $\mathbb{N} \setminus \{0\}$, then compute all subsets $C\subseteq \mathbb{N}\setminus \{0\}$ that are minimal (with respect to the inclusion order) in the family of sets $K\subseteq \mathbb{N}\setminus \{0\}$ that satisfy the following properties:
	\begin{enumerate}[label=(\roman*)]
		\item $K\cap \langle A \rangle =\emptyset$.
		\item If $x_1+\cdots+x_n$ is a partition of an element of $B$, then at least one of the summands of the partition belongs to $K$.
	\end{enumerate}
\end{problem}
  
To achieve our aim, we use the theory of numerical semigroups. In particular, we consider the family of irreducible numerical semigroups and the concept of Ap\'ery set of a numerical semigroup. In Section~\ref{preliminar} we recall some preliminaries on numerical semigroups and, moreover, we relate Problem~\ref{prob} to a question on numerical semigroups.

Let $\mathscr{S}$ be the set of all numerical semigroups. If $A$ is a non-empty subset of $\mathbb{N}$, $F\in\mathbb{N}\setminus\{0\}$, and $\mathrm{F}(S)$ is the Frobenius number of a numerical semigroup $S$ (see Section~\ref{preliminar}, first paragraph), then we denote by
\[ \mathscr{S}(A,F) = \{ S\in\mathscr{S} \mid A\subseteq S \mbox{ and } \mathrm{F}(S)=F \} \]
and by
\[ \mathscr{I}(A,F) = \{ S\in\mathscr{S}(A,F) \mid S \mbox{ is irreducible} \}. \]
In Section~\ref{sect-I(A,F)}, we show that $\mathscr{I}(A,F)\not=\emptyset$ if and only if $\mathscr{S}(A,F)\not=\emptyset$ if and only if $F\notin\langle A \rangle$.

In \cite{forum} it is described an algorithmic process, implemented in the function \textsf{IrreducibleNumericalSemigroupWithFrobeniusNumber($\cdot$)} of \cite{numericalsgps}, which computes all the irreducible numerical semigroups with a given Frobenius number.

Let us observe that we can compute all the elements of $\mathscr{I}(A,F)$ quite simply: by using \textsf{IrreducibleNumericalSemigroupWithFrobeniusNumber$(F)$} we obtain all the irreducible numerical semigroups with Frobenius number equal to $F$ and, after that, we remove those that do not contain the set $A$.

The main purpose of Section~\ref{sect-I(A,F)} is to show an algorithm for computing $\mathscr{I}(A,F)$ that improves the one described above.

As a first application of the above algorithm, in Section~\ref{sect-S(A,F)} we give one to compute $\mathscr{S}(A,F)$. Such an algorithm is a generalisation of the one shown in \cite{computer} for computing all the numerical semigroups with as given Frobenius number.

Let us observe that Section~\ref{sect-S(A,F)} can be considered a by-product result and, in a first reading, it can be omitted without losing the thread.

If $A$ is a non-empty subset of $\mathbb{N}$ and $B$ is a non-empty finite subset of $\mathbb{N} \setminus \{0\}$, then we denote by
\[ \mathscr{S}(A,B) = \{ S \in \mathscr{S} \mid A\subseteq S \mbox{ and } S\cap B=\emptyset \}. \]
We begin Section~\ref{sect-M(A,B)} by showing that $\mathscr{S}(A,B)\not=\emptyset$ if and only if $B\cap\langle A \rangle=\emptyset$. Denoting by
\[ \mathscr{M}(A,B) = \mathrm{Maximal} \big(\mathscr{S}(A,B)\big), \]
in Section~\ref{sect-M(A,B)} we use the results of Section~\ref{sect-I(A,F)} in order to show an algorithm for computing $\mathscr{M}(A,B)$. Such an algorithm is a generalisation of the one shown in \cite{royal} for computing the maximal elements of $\{ S \in \mathscr{S} \mid S \cap B = \emptyset \}$.

Finally, in Section~\ref{sect-problem} we show an algorithm as an answer to Problem~\ref{prob}. As one might expect, the algorithm described is a generalisation of the one shown in \cite{royal}.


\section{Preliminaries}\label{preliminar}

A \emph{submonoid} of $(\mathbb{N},+)$ is a subset $M\subseteq\mathbb{N}$ that is closed under addition and contains the zero element. A \emph{numerical semigroup} is a submonoid $S$ of $\mathbb{N}$ such that $\mathbb{N} \setminus S$ is finite. The finiteness of $\mathbb{N} \setminus S$ allows us to define two relevant invariants of $S$. Namely, the greatest integer that does not belong to $S$, called the \emph{Frobenius number} of $S$ and denoted by $\mathrm{F}(S)$, and the cardinality of $\mathbb{N} \setminus S$, called the \emph{genus} of $S$ and denoted by $\mathrm{g}(S)$.

Following the notation of \cite{pacific}, a numerical semigroup is \emph{irreducible} if it can not be expressed as the intersection of two numerical semigroups containing it strictly. In \cite{pacific} it is shown that a numerical semigroup $S$ is irreducible if and only if is maximal (with respect the inclusion order) in the set of all numerical semigroups with Frobenius number equal to $\mathrm{F}(S)$. Moreover, from \cite{barucci} and \cite{froberg}, we can state that the set of all irreducible numerical semigroups is the union of two interesting families of numerical semigroups: the symmetric and the pseudo-symmetric ones.

If $X$ is a non-empty subset of $\mathbb{N}$, then we denote by $\langle X \rangle$ the submonoid of $(\mathbb{N},+)$ generated by $X$, that is,
\[ \langle X \rangle=\big\{\lambda_1x_1+\cdots+\lambda_nx_n \mid n\in\mathbb{N}\setminus \{0\}, \ \{x_1,\ldots,x_n\}\subseteq X, \ \{\lambda_1,\ldots,\lambda_n\}\subseteq \mathbb{N}\big\}. \]
It is well known (see Lemma~2.1 of \cite{springer}) that $\langle X \rangle$ is a numerical semigroup if and only if $\gcd(X)=1$.

If $M$ is a submonoid of $(\mathbb{N},+)$ and $M=\langle X \rangle$, then we say that $X$ is a \emph{system of generators} of $M$. Moreover, if $M\not=\langle Y \rangle$ for any subset $Y\subsetneq X$, then we say that $X$ is a \emph{minimal system of generators} of $M$.

The following result is \cite[Corollary~2.8]{springer}.

\begin{lemma}\label{lem-6}
	If $M$ is a	submonoid of $(\mathbb{N},+)$, then $M$ has a unique minimal system of generators that in addition is finite.
\end{lemma}

We denote by $\mathrm{msg}(M)$ the minimal system of generators of $M$. The next result is \cite[Corollary~2.9]{springer}.

\begin{lemma}\label{lem-7}
	If $M$ is a	submonoid of $(\mathbb{N},+)$ generated by $\{m_1<\cdots<m_e\}$ (with $m_1\not=0$), then $\mathrm{msg}(M)=\{m_1<\cdots<m_e\}$ if and only if $m_{i+1}\notin\langle m_1,\ldots,m_i \rangle$ for all $i\in\{1,\ldots,e-1\}$.
\end{lemma}

Let $S$ be a numerical semigroup and $n\in S\setminus\{0\}$. The \emph{Ap\'ery set of $n$ in $S$} (named so in honour of \cite{apery}) is $\mathrm{Ap}(S,n)=\{s\in S\mid s-n\notin S\}$. The next result is \cite[Lemma 2.4]{springer}.

\begin{lemma}\label{lem-29}
	Let $S$ be a numerical semigroup and let $n\in S\setminus\{0\}$. Then $\mathrm{Ap}(S,n)=\{w(0)=0,w(1),\ldots,w(n-1)\}$, where $w(i)$ is the least element of $S$ congruent with $i$ modulo $n$, for all $i\in\{0,\ldots,n-1\}$.
\end{lemma}

If $S$ is a numerical semigroup and $\mathrm{Ap}(S,n)=\{w(0),w(1),\ldots,w(n-1)\}$, then we denote by $\theta_n(S)=(w(1),\ldots,w(n-1))$. Moreover, if $(x_1,\ldots,x_r)$ and $(y_1,\ldots,y_r)$ belong to $\mathbb{N}^r$, then we define the operation \[(x_1,\ldots,x_r)\lor(y_1,\ldots,y_r)=\left(\max\{x_1,y_1\},\ldots,\max\{x_r,y_r\}\right)\]
and we consider the usual product order
\[(x_1,\ldots,x_r)\leq(y_1,\ldots,y_r) \mbox{ if } x_i\leq y_i \mbox{ for all } i\in\{1,\ldots,r\}.\]

The following result is \cite[Corollary 4.5]{royal} and useful for proving Theorem~\ref{thm-31}.

\begin{proposition}\label{prop-30}
	Let $n$ be a positive integer and let $\mathscr{S}_n = \{ S\in\mathscr{S} \mid n\in S \}$. Then:
	\begin{enumerate}
		\item $\theta_n:\mathscr{S}_n \to \mathbb{N}^{n-1}$ is an injective application.
		\item If $(x_1,\ldots,x_{n-1})\in\mathrm{Im}(\theta_n)$, then $S=\langle x_1,\ldots,x_{n-1},n \rangle \in \mathscr{S}_n$ and $\theta_n(S)=(x_1,\ldots,x_{n-1})$.
		\item If $\{S,T\}\subseteq\mathscr{S}_n$, then $S\subseteq T$ if and only if $\theta_n(T)\leq\theta_n(S)$.
		\item If $\{S,T\}\subseteq\mathscr{S}_n$, then $\theta_n(S\cap T)=\theta_n(S)\lor\theta_n(T)$.
	\end{enumerate}
\end{proposition}

Now, let $A$ be a non-empty subset of $\mathbb{N}$ and let $B$ be a non-empty finite subset of $\mathbb{N} \setminus \{0\}$ such that $\langle A \rangle \cap B = \emptyset$. Let us consider the notation
\begin{itemize}
	\item $\mathfrak{P}(B) = \{ (x_1,\ldots,x_n) \mid x_1+\cdots+x_n \mbox{ is a partition of some element of } B\}$,
	\item $\mathfrak{G}(A,B) = \{ \{x_1,\ldots,x_n\}\setminus \langle A \rangle \mid (x_1,\ldots,x_n) \in \mathfrak{P}(B) \}$,
	\item $\mathfrak{L}(A,B) = \{ K \subseteq \mathbb{N} \setminus \{0\} \mid K\cap X \not= \emptyset \mbox{ for all } X\in\mathfrak{G}(A,B) \}$,
	\item $\mathfrak{m}(A,B) =  \mathrm{Minimal}\big( \mathfrak{L}(A,B) \big)$.
\end{itemize}

Observe that, with this notation, the question proposed at Problem~\ref{prob} is equivalent to give an algorithm for computing $\mathfrak{m}(A,B)$. 

Our next purpose is to prove that $C \in \mathfrak{m}(A,B)$ if and only if $\mathbb{N}\setminus C \in \mathscr{M}(A,B)$. To do so, we need the following three results.

\begin{lemma}\label{lem-33}
	If $C \in \mathfrak{m}(A,B)$, then the following conditions are satisfied.
	\begin{enumerate}
		\item $B \subseteq C \subseteq \{1,\ldots,\max(B)\}\setminus \langle A \rangle$. \label{lem-33-1}
		\item If $c\in C$, then there exists $X\in\mathfrak{G}(A,B)$ such that $C\cap X=\{c\}$.
		\item If $x,y$ are positive integers and $x+y\in C$, then $C\cap \{x,y\}\not=\emptyset$.
		\item If $\{x,y\}\subseteq\mathbb{N}\setminus C$, then $x+y \in \mathbb{N}\setminus C$. \label{lem-33-4}
	\end{enumerate}
\end{lemma}

\begin{proof}
	\mbox{ }
	\begin{enumerate}
		\item It is trivial from the definitions.
		\item If $C \in \mathfrak{m}(A,B)$ and $c\in C$, then $C\setminus\{c\} \not\in \mathfrak{L}(A,B)$. Therefore, there exits $X\in \mathfrak{G}(A,B)$ such that $(C\setminus\{c\}) \cap X = \emptyset$. Moreover, since $C \in \mathfrak{m}(A,B)$, then $C \in \mathfrak{L}(A,B)$ and, consequently, $C\cap X \not= \emptyset$. Thereby, we conclude that $C\cap X=\{c\}$.
		\item Let $x,y$ be positive integers such $x+y\in C$. By applying the previous item, there exists $(x_1,\ldots,x_n) \in \mathfrak{P}(B)$ such that $x_1+\cdots+x_n = b\in B$ and
		\begin{equation}\label{eq-1}
			C \cap\left( \{x_1,\ldots,x_n\} \setminus \langle A \rangle \right) = \{x+y\}.
		\end{equation}
		Now, without loss of generality, we can suppose that $x_1=\cdots=x_k=x+y$ and $x_i\not=x+y$ for all $i\in\{k+1,\ldots,n\}$. Thus, $(x+y)+\stackrel{(k)}{\cdots}+(x+y)+x_{k+1}+\cdots+x_n=b$, that is, $\{x,y,x_{k+1},\ldots,x_n\} \setminus \langle A \rangle \in \mathfrak{G}(A,B)$ and, since $C\in\mathfrak{L}(A,B)$, we have that
		\begin{equation}\label{eq-2}
			C \cap \big( \{x,y,x_{k+1},\ldots,x_n\} \setminus \langle A \rangle \big) \not= \emptyset.
		\end{equation}
		From \eqref{eq-1}-\eqref{eq-2} (and being $x_i\not=x+y$ if $k+1 \leq i \leq n\}$), it is clear that $C\cap \{x,y\}\not=\emptyset$.
		\item This is a reformulation of the statement in the previous item. \qedhere
	\end{enumerate}	
\end{proof}

\begin{proposition}\label{prop-34}
	If $C\in \mathfrak{m}(A,B)$, then $\mathbb{N} \setminus C \in \mathscr{S}(A,B)$. Moreover, we have that $\langle A \rangle \subseteq \mathbb{N}\setminus C$ and $ (\mathbb{N}\setminus C) \cap B = \emptyset$.
\end{proposition}

\begin{proof}
	First of all, by Item~\ref{lem-33-1} of Lemma~\ref{lem-33}, we know that  $\mathbb{N} \setminus C$ is a finite set such that $0 \in \mathbb{N} \setminus C$. Secondly, by Item~\ref{lem-33-4} of Lemma~\ref{lem-33}, we have that $\mathbb{N} \setminus C$ is closed under addition. Therefore, $\mathbb{N} \setminus C$ is a numerical semigroup.
	
	Finally, and again by Item~\ref{lem-33-1} of Lemma~\ref{lem-33}, we can conclude that  $\langle A \rangle \subseteq \mathbb{N} \setminus C$ and $(\mathbb{N} \setminus C) \cap B = \emptyset$.
\end{proof}

\begin{proposition}\label{prop-35}
	If $S$ is a numerical semigroup such that $A\subseteq S$ and $S\cap B=\emptyset$, then $\mathbb{N}\setminus S \in \mathfrak{L}(A,B)$.
\end{proposition}

\begin{proof}
	If we take $K=\mathbb{N}\setminus S$, then it is clear that $K \subseteq \mathbb{N}\setminus \{0\}$. Thus, to finish the proof, it is enough to see that, if $X\in\mathfrak{G}(A,B)$, then $K\cap X\not=\emptyset$.
	
	Let $X = \{x_1,\ldots,x_n\}\setminus  \langle A \rangle \in \mathfrak{G}(A,B)$. By hypothesis, if $x_1+\cdots+x_n$ is a partition of $b\in B$, then $x_1+\cdots+x_n\not\in S$. Thus, there exists $i\in\{1,\ldots,n\}$ such that $x_i\not\in S$, that is, $x_i\in\mathbb{N}\setminus S=K$. In addition, $x_i\not\in \langle A \rangle$ and, consequently, $x_i\in X$. Thereby, $x_i\in K\cap X$.
\end{proof}

We are now ready to show the result that allows to translate Problem~\ref{prob} into a question on numerical semigroups.

\begin{theorem}\label{thm-36}
	Let $C$ be a subset of $\mathbb{N} \setminus \{0\}$. Then $C\in \mathfrak{m}(A,B)$ if and  only if $\mathbb{N} \setminus C \in \mathscr{M}(A,B)$.
\end{theorem}

\begin{proof}
	(\emph{Necessity}.) From Proposition~\ref{prop-34}, we know that $\mathbb{N} \setminus C\in \mathscr{S}(A,B)$. Now, let us suppose that there exists $S\in\mathscr{S}(A,B)$ such that $\mathbb{N} \setminus C \subseteq S$. Then, by Proposition~\ref{prop-35}, we have that $\mathbb{N} \setminus S \in\mathfrak{L}(A,B)$. But, since $C\in \mathfrak{m}(A,B)$ and $\mathbb{N} \setminus S \subseteq C$, we have that $C=\mathbb{N} \setminus S$ or, equivalently, that $\mathbb{N} \setminus C=S$. Thus, we conclude that $\mathbb{N} \setminus C \in \mathscr{M}(A,B)$.
	
	(\emph{Sufficiency}.) By applying Proposition~\ref{prop-35}, we have that, if $\mathbb{N} \setminus C \in \mathscr{M}(A,B)$, then $C\in\mathfrak{L}(A,B)$. Now, let us assume that there exists $D\in\mathfrak{m}(A,B)$ such that $D\subseteq C$. Then, $\mathbb{N}\setminus C \subseteq \mathbb{N}\setminus D$ and, by Proposition~\ref{prop-34}, $\mathbb{N}\setminus D \in  \mathscr{S}(A,B)$. Therefore, $\mathbb{N}\setminus C = \mathbb{N}\setminus D$ and, consequently, $C=D$. Thus, $C\in \mathfrak{m}(A,B)$.
\end{proof}

As an immediate consequence of Theorem~\ref{thm-36}, we have the following result.

\begin{corollary}\label{cor-37}
	Let $A$ be a non-empty subset of $\mathbb{N}$ and let $B$ be a non-empty finite subset of $\mathbb{N} \setminus \{0\}$. Then $\mathfrak{m}(A,B)=\left\{ \mathbb{N}\setminus S \mid S \in \mathscr{M}(A,B) \right\}$.
\end{corollary}

\section{The set $\mathscr{I}\boldsymbol{(A,F)}$}\label{sect-I(A,F)}

In this section $A$ is a non-empty subset of $\mathbb{N}$ and $F$ is a positive integer.

\begin{proposition}\label{prop-1}
	$\mathscr{S}(A,F)\not=\emptyset$ if and only if $F\notin\langle A \rangle$.
\end{proposition}

\begin{proof}
	(\emph{Necessity}.) If $\mathscr{S}(A,F)\not=\emptyset$, then there exists $S\in\mathscr{S}(A,F)$. Therefore, $\langle A \rangle \subseteq S$ and $F\notin S$. Consequently, $F\notin\langle A \rangle$.
	
	(\emph{Sufficiency}.) If $F\notin\langle A \rangle$, then it is clear that $S=\langle A \rangle \cup \{F+1,\to\}$ (where the symbol $\to$ means that every integer greater than $F+1$ belongs to the set $S$) is a numerical semigroup with Frobenius number equal to $F$. Therefore, $\mathscr{S}(A,F)\not=\emptyset$.
\end{proof}

The next result is \cite[Theorem 1]{pacific}.

\begin{lemma}\label{lem-2}
	If $S$ is a numerical semigroup, then the following conditions are equivalent.
	\begin{enumerate}
		\item $S$ is irreducible.
		\item $S$ is maximal in the set of all numerical semigroups with Frobenius number equal to $\mathrm{F}(S)$.
		\item $S$ is maximal in the set of all numerical semigroups that do not contain $\mathrm{F}(S)$.
	\end{enumerate}
\end{lemma}

Let us observe that, if $S$ is a numerical semigroup, then $\mathbb{N} \setminus S$ is a finite set. Therefore, there exists a finite number of numerical semigroups containing $S$. This fact, together with Proposition~\ref{prop-1} and Lemma~\ref{lem-2}, allows us to state the next result.

\begin{proposition}\label{prop-3}
	$\mathscr{I}(A,F)\not=\emptyset$ if and only if $F\notin\langle A \rangle$.
\end{proposition}

The following result is \cite[Lemma 4]{computer}.

\begin{lemma}\label{lem-4}
	If $S$ is a numerical semigroup with Frobenius number equal to $F$, then:
	\begin{enumerate}
		\item If $h = \max \left\{ x\in\mathbb{N}\setminus S \mid F-x\notin S \mbox{ and } x\not=\frac{F}{2} \right\}$, then $S\cup\{h\}$ is a numerical semigroup with Frobenius number equal to $F$.
		
		\item $S$ is irreducible if and only if $\left\{ x\in\mathbb{N}\setminus S \mid F-x\notin S \mbox{ and } x\not=\frac{F}{2} \right\} = \emptyset$.
	\end{enumerate}
\end{lemma}

From the above lemma, we can easily deduce the next result.

\begin{lemma}\label{lem-5}
	If $S$ is a numerical semigroup with Frobenius number equal to $F$, then
	\[ \Delta(S) = S \cup \left\{ x\in\mathbb{N}\setminus S \mid F-x\notin S \mbox{ and } x>\frac{F}{2} \right\} \]
	is an irreducible numerical semigroup with Frobenius number equal to $F$.
\end{lemma}

From the proof of Proposition~\ref{prop-1}, we know that, if $F\notin\langle A \rangle$, then $\langle A \rangle \cup \{F+1,\to\} \in \mathscr{S}(A,F)$. Now, by applying Lemma~\ref{lem-5}, we have that $\mathrm{C}(A,F)=\Delta(\langle A \rangle \cup \{F+1,\to\}) \in \mathscr{I}(A,F)$.

The following result says that $\mathrm{C}(A,F)$ is the unique element of $\mathscr{I}(A,F)$ in which all its minimal generators less than $\frac{F}{2}$ belong to $A$.

\begin{proposition}\label{prop-8}
	If $F\notin\langle A \rangle$, then
	\[ \Big\{ S\in\mathscr{I}(A,F) \mid \left\{ x\in\mathrm{msg}(S) \mid {\textstyle x<\frac{F}{2}} \right\} \subseteq A \Big\} = \{\mathrm{C}(A,F)\}. \]
\end{proposition}

\begin{proof}
	If $x \in \mathrm{C}(A,F)$ and $x<\frac{F}{2}$, then $x\in\langle A \rangle$. Since $\mathrm{C}(A,F) \in \mathscr{I}(A,F)$, by applying Lemma~\ref{lem-7}, we have that
	\[ \{\mathrm{C}(A,F)\} \subseteq \Big\{ S\in\mathscr{I}(A,F) \mid \left\{ x\in\mathrm{msg}(S) \mid {\textstyle x<\frac{F}{2}} \right\} \subseteq A \Big\}. \]
	
	If $S\in\mathscr{I}(A,F)$ and $\left\{ x\in\mathrm{msg}(S) \mid x<\frac{F}{2} \right\} \subseteq A$, then we deduce from Lemma~\ref{lem-7} that $S=\left( \langle A \rangle \cup \{F+1,\to\} \right) \cup X$ where $X\subseteq\{ x\in\mathbb{N} \mid \frac{F}{2} < x < F \}$. Now, by applying item 2 of Lemma~\ref{lem-4}, we easily get that $S=\mathrm{C}(A,F)$.
\end{proof}

The next result is well known (for instance, see \cite{springer}).

\begin{lemma}\label{lem-9}
	If $S$ is a numerical semigroup and $x\in S$, then $S\setminus \{x\}$ is a numerical semigroup if and only if $x\in\mathrm{msg}(S)$.
\end{lemma}

Let $\mathbb{Z}$ and $\mathbb{Q}$ be the sets of integer and rational numbers, respectively. If $q\in\mathbb{Q}$, we denote by $\lceil q \rceil = \min\{z\in \mathbb{Z} \mid q\leq z \}$. The following result can be deduced from \cite[Lemma 2.4]{forum}.

\begin{lemma}\label{lem-10}
	If $S$ is a numerical semigroup, then $S$ is irreducible if and only if $\mathrm{g}(S) = \left\lceil \frac{\mathrm{F}(S)+1}{2} \right\rceil$.
\end{lemma}

Note that, as a consequence of Proposition~\ref{prop-8}, if $S\in \mathscr{I}(A,F)\setminus \{\mathrm{C}(A,F)\}$, then there exists $\alpha(S)=\min\left\{x\in\mathrm{msg}(S) \mid x\notin A \mbox{ and } x<\frac{F}{2} \right\}$.

\begin{proposition}\label{prop-11}
	If $S\in \mathscr{I}(A,F)\setminus \{\mathrm{C}(A,F)\}$, then $\bar{S}=\left( S\setminus\{\alpha(S)\} \right) \cup \{F-\alpha(S)\} \in \mathscr{I}(A,F)$. 
\end{proposition}

\begin{proof}
	First of all, let us observe that the cardinality of $\mathbb{N}\setminus\bar{S}$ is equal to $\mathrm{g}(S)$. Therefore, from Lemma~\ref{lem-10}, we can state that $\bar{S}\in\mathscr{I}(A,F)$ if and only if $\bar{S}$ is a numerical semigroup.
	
	Now, from Lemma~\ref{lem-9}, we have that $S\setminus\{\alpha(S)\}$ is a numerical semigroup. Thus, in order to assert that $\bar{S}$ is a numerical semigroup, it suffices to verify that $F-\alpha(S)+s\in\bar{S}$ for $s=F-\alpha(S)$ or for all $s\in S\setminus\{0,\alpha(S)\}$.
	
	Firstly, and from its definition, $\alpha(S)<\frac{F}{2}$ and then $F-\alpha(S)>\frac{F}{2}$. Thereby, $2(F-\alpha(S))>F$ and, consequently, $2(F-\alpha(S))\in\bar{S}$.
	
	Secondly, if $s\in S\setminus\{0,\alpha(S)\}$ and $F-\alpha(S)+s\notin S$, then, by item 2 of Lemma~\ref{lem-4}, we deduce that $F-\left( F-\alpha(S)+s \right)\in S$ or $F-\alpha(S)+s=\frac{F}{2}$. But, if $F-\left( F-\alpha(S)+s \right)\in S$, then $\alpha(S)-s\in S$ and, in consequence, $\alpha(S)\notin\mathrm{msg}(S)$, which is a contradiction. And, if $F-\alpha(S)+s=\frac{F}{2}$, then $\alpha(S)=\frac{F}{2}+s>\frac{F}{2}$, which is once again a contradiction.
\end{proof}

If $S\in \mathscr{I}(A,F)$, then Proposition~\ref{prop-11} allows us to define the following sequence of elements in $\mathscr{I}(A,F)$.
\begin{itemize}
	\item $S_0=S$;
	\item $S_{n+1}=\left\{ \begin{array}{l} \left( S_n \setminus \{\alpha(S_n)\} \right) \cup \{F-\alpha(S_n)\}, \mbox{ if } S_n\not=\mathrm{C}(A,F), \\[3pt] \mathrm{C}(A,F), \mbox{ in other case.} \end{array} \right.$
\end{itemize}
It is clear that there exists $k\in\mathbb{N}$ such that $S_k=\mathrm{C}(A,F)$.

In order to compute the elements of $\mathscr{I}(A,F)$, we use a tree structure on that set. To do so, let us recall some definitions and results.

A \emph{graph} $G$ is a pair $(V,E)$ where $V$ is a non-empty set (whose elements are called \emph{vertices} of $G$) and $E$ is a subset of $\{(v,w) \in V \times V \mid v \neq w\}$ (whose elements are called \emph{edges} of $G$). A \emph{path (of length $n$) connecting the vertices $x$ and $y$ of $G$} is a sequence of different edges $(v_0,v_1),(v_1,v_2),\ldots,(v_{n-1},v_n)$ such that $v_0=x$ and $v_n=y$. We say that a graph $G$ is a \emph{(directed rooted) tree} if there exists a vertex $r$ (known as the \emph{root} of $G$) such that, for any other vertex $x$ of $G$, there exists a unique path connecting $x$ and $r$. If $(x,y)$ is an edge of the tree, then we say that $x$ is a \emph{child} of $y$. (For more details about trees, see \cite[Chapter 9]{rosen}.)

We define the graph $\mathrm{G}\big(\mathscr{I}(A,F)\big)$ in the following way: $\mathscr{I}(A,F)$ is the set of vertices and $(S,T)\in \mathscr{I}(A,F) \times \mathscr{I}(A,F)$ is an edge if $T=\left( S \setminus \{\alpha(S)\} \right) \cup \{F-\alpha(S)\}$.

The next result is an immediate consequence of the definition of the sequence $\{S_n\mid n\in\mathbb{N}\}\subseteq \mathscr{I}(A,F)$ given after Proposition~\ref{prop-11}.

\begin{theorem}\label{thm-12}
	If $F\notin\langle A \rangle$, then $\mathrm{G}\big(\mathscr{I}(A,F)\big)$ is a tree with root $\mathrm{C}(A,F)$.
\end{theorem}

It is clear that a tree can be built in a recurrent way by starting from its root and connecting each vertex with its children. Our next purpose is to show who are the children of a vertex of $\mathrm{G}\big(\mathscr{I}(A,F)\big)$.

The following result is \cite[Proposition 2.5]{forum}.

\begin{lemma}\label{lem-13}
	Let $S$ be an irreducible numerical semigroup with Frobenius number equal to $F$ and let $x$ be an element of $\mathrm{msg}(S)$ such that $x<F$, $2x-F\notin S$, $3x\not=2F$, and $4x\not=3F$. Then $\bar{S}=\left( S \setminus \{x\} \right) \cup \{F-x\}$ is another irreducible numerical semigroup with Frobenius number equal to $F$.
\end{lemma}

It is obvious that $\left\{ x\in\mathrm{msg}\left(\mathrm{C}(A,F)\right) \mid x\notin A \mbox{ and } x<\frac{F}{2} \right\} = \emptyset$. Thus, there does not exist $\min\left\{ x\in\mathrm{msg}\left(\mathrm{C}(A,F)\right) \mid x\notin A \mbox{ and } x<\frac{F}{2} \right\}$. By definition, we consider that $\alpha\left(\mathrm{C}(A,F)\right)=+\infty$.

Let us build the children set of a numerical semigroup $S$ belonging to the tree $\mathrm{G}\big(\mathscr{I}(A,F)\big)$.

\begin{proposition}\label{prop-14}
	If $S\in\mathscr{I}(A,F)$, then the children set of $S$, in the tree $\mathrm{G}\big(\mathscr{I}(A,F)\big)$, is
	\begin{eqnarray}
		\Big\{ \big( S \setminus \{x\} \big) \cup \{F-x\} \mid x\in\mathrm{msg}(S), \; \frac{F}{2}<x<F, \; x\notin A, \; 2x-F\notin S, \nonumber \\ 3x\not=2F, \; 4x\not=3F, \, \mbox{and} \; F-x<\alpha(S) \Big\}. \nonumber
	\end{eqnarray}
\end{proposition}

\begin{proof}
	Let  $x\in\mathrm{msg}(S)$ such that $\frac{F}{2}<x<F$, $x\notin A$, $2x-F\notin S$, $3x\not=2F$, $4x\not=3F$, and $F-x<\alpha(S)$. By Lemma~\ref{lem-13}, we have that $T=\left( S \setminus \{x\} \right) \cup \{F-x\} \in \mathscr{I}(A,F)$. Moreover, since $F-x<\alpha(S)$, then $\alpha(T)=F-x$ and, as a result, $S=\left( T \setminus \{\alpha(T)\} \right) \cup \{F-\alpha(T)\}$. Consequently, and having in mind that $\alpha(T)\in\mathrm{msg}(T)$, we conclude that $T$ is a child of $S$.
	
	Let us see the other inclusion. If $T$ is a child of $S$, then $S=\left( T \setminus \{\alpha(T)\} \right) \cup \{F-\alpha(T)\}$. Therefore, $T=\left( S \setminus \{F-\alpha(T)\} \right) \cup \{F-\left(F-\alpha(T)\right)\}$. It is clear that $F-\alpha(T)\in A$, $\frac{F}{2}<F-\alpha<F$, and $F-\left(F-\alpha(T)\right)<\alpha(S)$. In addition, 
	\begin{itemize}
		\item if $F-\alpha(T)\in A$, then $F\in T$;
		\item if $2(F-\alpha(T))-F\in S$, then $F-2\alpha(T) \in S$ and, since $2\alpha(T)\in S$, we have that $F\in S$;
		\item if $3(F-\alpha(T))-F=2F$, then $F=3\alpha(S)\in S$;
		\item if $4(F-\alpha(T))-F=3F$, then $F=4\alpha(S)\in S$;
	\end{itemize}
	which are four contradictions and, thus, the proof is complete.
\end{proof}

We are now in a position to show the algorithm for computing $\mathscr{I}(A,F)$.

\begin{algorithm}\label{alg-15}
	\mbox{ } \par	
	\noindent $\ $ INPUT: $A\subseteq\mathbb{N}$ and $F\in\mathbb{N}\setminus\{0\}$ such that $F\notin\langle A \rangle$. \par \smallskip
	\noindent $\ $ OUTPUT: $\mathscr{I}(A,F)$.
	\vspace{-3pt}
	\begin{itemize}
		\setlength\itemsep{0pt}
		\item[(1)] $X = \{ \mathrm{C}(A,F) \}$.
		\item[(2)] $Y = \Big\{ \big( \mathrm{C}(A,F) \setminus \{x\} \big) \cup \{F-x\} \mid x\in\mathrm{msg}(\mathrm{C}(A,F)), \; \frac{F}{2}<x<F, \; x\notin A,$ \vspace{-12pt} \begin{flushright} $2x-F\notin \mathrm{C}(A,F), \; 3x\not=2F, \, \mbox{and} \; 4x\not=3F \Big\}$. \end{flushright} \vspace{-10pt}
		\item[(3)] If $Y=\emptyset$, then return $X$.
		\item[(4)] $X := X\cup Y$.
		\item[(5)] For each $S\in Y$, compute $Y_S = \Big\{ \big( S \setminus \{x\} \big) \cup \{F-x\} \mid x\in\mathrm{msg}(S),$ \vspace{-12pt} \begin{flushright} $ \frac{F}{2}<x<F, \; x\notin A, \; 2x-F\notin S, \; 3x\not=2F, \, 4x\not=3F, \; \mbox{and} \; F-x<\alpha(S) \Big\}.$ \end{flushright} \vspace{-10pt}
		\item[(6)] $Y:=\bigcup_{S\in Y} Y_s$ and go to $(3)$.
	\end{itemize}
\end{algorithm}

Let us illustrate how the above algorithm works by means of an example.

\begin{example}\label{exmp-16}
	Using Algorithm~\ref{alg-15}, let us compute $\mathscr{I}(\{4\},11)$.
	\begin{itemize}
		\setlength\itemsep{0pt}
		\item $X = \{ \mathrm{C}(\{4\},11) \} =\{ \langle 4,6,9 \rangle \}$.
		\item $Y = \left\{ \big( \langle 4,6,9 \rangle \setminus \{6\} \big) \cup \{11-6\}, \; \big( \langle 4,6,9 \rangle \setminus \{9\} \big) \cup \{11-9\} \right\} \Rightarrow$ \vspace{3pt} \newline $Y = \left\{ \langle 4,5 \rangle, \langle 2,13 \rangle \right\}$.
		\item $X = \left\{ \langle 4,6,9 \rangle, \langle 4,5 \rangle, \langle 2,13 \rangle \right\}$.
		\item $Y_{\langle 4,5 \rangle} = \emptyset$, $Y_{\langle 2,13 \rangle} = \emptyset$.
		\item $Y = \emptyset$.	
		\item Return $\mathscr{I}(\{4\},11) = \left\{ \langle 4,6,9 \rangle, \langle 4,5 \rangle, \langle 2,13 \rangle \right\}$.
	\end{itemize}
\end{example}

\begin{remark}\label{rem-15}
	Observe that, if we take $A=\{0\}$, then $\mathscr{I}(\{0\},F)$ is the set of the irreducible numerical semigroups with Frobenius number equal to $F$. Therefore, Algorithm~\ref{alg-15} is a generalisation of the algorithm described in \cite{forum} for computing the irreducible numerical semigroups with fixed Frobenius number.
\end{remark}

\section{The set $\mathscr{S}\boldsymbol{(A,F)}$}\label{sect-S(A,F)}

As in the previous section, in this one $A$ is a non-empty subset of $\mathbb{N}$ and $F$ is a positive integer. Our main purpose here is to show an algorithm that allows us to compute $\mathscr{S}(A,F)$.

Recall that, by Proposition~\ref{prop-1}, we know that $\mathscr{S}(A,F)\not=\emptyset$ if and only if $F\notin \langle A \rangle$. Moreover, by Lemma~\ref{lem-5}, we know that, if $S$ is a numerical semigroup with Frobenius number $F$, then $\Delta(S) = S \cup \left\{ x\in\mathbb{N}\setminus S \mid F-x\notin S \mbox{ and } x>\frac{F}{2} \right\}$ is an irreducible numerical semigroup with Frobenius number equal to $F$.

Let us define on $\mathscr{S}(A,F)$ the equivalence relation
\[ S \,\mathfrak{r}\, T  \,\mbox{ if }\,  \Delta(S)=\Delta(T). \]
If $S\in\mathscr{S}(A,F)$, then $[S] = \big\{ T\in\mathscr{S}(A,F) \mid S \,\mathfrak{r}\, T \big\}$ and it is well known that the quotient set $\mathscr{S}(A,F)/\mathfrak{r}=\big\{ [S] \mid S\in\mathscr{S}(A,F) \big\}$ is a partition of $\mathscr{S}(A,F)$.

In the following result we stablish that there is a one-to-one correspondence between the sets $\mathscr{I}(A,F)$ and $\mathscr{S}(A,F)/\mathfrak{r}$. 

\begin{proposition}\label{prop-17}
	If $F\notin\langle A \rangle$, then $\mathscr{S}(A,F)/\mathfrak{r}=\big\{ [S] \mid S\in\mathscr{I}(A,F) \big\}$. Moreover, if $\{S,T\}\subseteq\mathscr{I}(A,F)$ and $S\not=T$, then $[S]\cap[T]=\emptyset$.
\end{proposition}

\begin{proof}
	From Lemma~\ref{lem-5}, if $S\in\mathscr{S}(A,F)$, then $\Delta(S)\in\mathscr{I}(A,F)$. Moreover, it is clear that $\Delta(S)=\Delta(\Delta(S))$. Therefore, $\mathscr{S}(A,F)/\mathfrak{r}=\big\{ [S] \mid S\in\mathscr{I}(A,F) \big\}$.
	
	Now, if $\{S,T\}\subseteq\mathscr{I}(A,F)$ and $[S]\cap[T]\not=\emptyset$, then $\Delta(S)=\Delta(T)$. But, since $S$ and $T$ are irreducible, then $S=\Delta(S)$ and $T=\Delta(T)$. Thus, we conclude that $S=T$.
\end{proof}

As a consequence of Proposition~\ref{prop-17}, we have that, in order to build all the elements of $\mathscr{S}(A,F)$, it is sufficient to have
\begin{enumerate}[label=(\roman*)]
	\item an algorithm to build $\mathscr{I}(A,F)$ and \label{en-17-1}
	\item an algorithm that, given $S\in\mathscr{I}(A,F)$, build $[S]$. \label{en-17-2}
\end{enumerate}
Since Algorithm~\ref{alg-15} solves \ref{en-17-1}, we focus on our attention in describing a procedure for solving \ref{en-17-2}.

The next result has an easy proof and is therefore omitted.

\begin{lemma}\label{lem-18}
	If $S\in\mathscr{I}(A,F)$, then the following facts are true.
	\begin{enumerate}
		\item $\max([S])=S$ and $\,\min([S])=\left\langle A\cup\left\{x\in S \mid x<\frac{F}{2}\right\}\right\rangle \cup \{F+1,\to\}$.
		
		\item If $T$ is a numerical semigroup, then $T\in[S]$ if and only if $\min([S])\subseteq T \subseteq S$.
	\end{enumerate}
\end{lemma}

If $S\in\mathscr{I}(A,F)$, then we denote by $\Gamma(S)=\min([S])$ and by $\mathcal{D}(S)=S\setminus\Gamma(S)$. Moreover, as usual, if $A,B\subseteq \mathbb{Z}$, then $A+B = \{a+b \mid a\in A, \, b\in B\}$.

In the next result, fixed $S\in\mathscr{I}(A,F)$, we describe all the elements of $[S]$.

\begin{proposition}\label{prop-19}
	Let $S\in\mathscr{I}(A,F)$ and let $B$ be a subset of $\mathcal{D}(S)$. Then $\Gamma(S) \cup \left( \big(B+\Gamma(S)\big) \cap \mathcal{D}(S) \right) \in[S]$. Moreover, all the elements of $[S]$ are built in this way.
\end{proposition}

\begin{proof}
	Let $\bar{S} = \Gamma(S) \cup \left( \big(B+\Gamma(S)\big) \cap \mathcal{D}(S) \right)$. Then $\Gamma(S) \subseteq \bar{S} \subseteq S$ and therefore, from Lemma~\ref{lem-18}, if we see that $\bar{S}$ is a numerical semigroup, then we will get that $\bar{S}\in[S]$.
	
	First of all, since $\Gamma(S) \subseteq \bar{S}$, we have that $0 \in \bar{S}$ and $\mathbb{N} \setminus \bar{S} \subseteq \mathbb{N} \setminus \Gamma(S)$. Thus, $0 \in \bar{S}$ and $\mathbb{N} \setminus \bar{S}$ is finite.
	
	Secondly, we have to see that $\bar{S}$ is closed under addition. For this we have three possibilities.
	\begin{itemize}
		\item If $x,y \in \Gamma(S)$, then $x+y \in \Gamma(S) \subseteq \bar{S}$.
		\item If $x,y \in \big(B+\Gamma(S)\big) \cap \mathcal{D}(S)$, then $x,y \in \mathcal{D}(S)$ and, thereby, $x,y > \frac{F}{2}$. Thus, $x+y > F$ and, consequently, $x+y \in \Gamma(S) \subseteq \bar{S}$.
		\item Let us take $x \in \Gamma(S)$ and $y \in \big(B+\Gamma(S)\big) \cap \mathcal{D}(S)$. Then we have that $y=y_1+y_2$, with $y_1\in B$ and $y_2\in\Gamma(S)$, and, therefore, $x+y \in B+\Gamma(S)$. Now,
		\begin{itemize}[label=$\ast$]
			\item on the one hand, if $x+y \in \mathcal{D}(S)$, then $x+y \in \big(B+\Gamma(S)\big) \cap \mathcal{D}(S) \subseteq \bar{S}$,
			\item and, on the other hand, if $x+y \not\in \mathcal{D}(S)$, then $x+y \in \Gamma(S) \subseteq \bar{S}$ (observe that $x+y\in S$ because $x\in\Gamma(S)\subseteq S$ and $y\in\mathcal{D}(S)\subseteq S$).
		\end{itemize}
	\end{itemize}

	Now, let $T$ be a numerical semigroup such that $T\in [S]$. From Lemma~\ref{lem-18}, we have that $\Gamma(S)\subseteq T \subseteq S$ and, therefore, $T = \Gamma(S) \cup B$ for some set $B \subseteq \mathcal{D}(S)$. Then, having in mind that $T$ is a numerical semigroup, we get that $T=\Gamma(S) \cup \left( \big(B+\Gamma(S)\big) \cap \mathcal{D}(S) \right)$. Effectively, it is clear that $B \subseteq \big(B+\Gamma(S)\big) \cap \mathcal{D}(S)$, but
	\[ \big(B+\Gamma(S)\big) \cap \mathcal{D}(S) \subseteq T \cap \mathcal{D}(S) = 
	\big( \Gamma(S) \cup B \big) \cap \mathcal{D}(S) = \]
	\[ \big( \Gamma(S) \cap \mathcal{D}(S) \big) \cup \big( B \cap \mathcal{D}(S) \big) = \emptyset \cup B = B. \]
	Thus, we finish the proof.
\end{proof}

Let us note that, if $\Gamma(S)=S$ for some $S\in\mathscr{I}(A,F)$, then $\mathcal{D}(S) = \emptyset$. In other case, if $d\in\mathcal{D}(S)$, then we denote by $\mathrm{T}(d) = \left( \{d\}+\Gamma(S) \right) \cap \mathcal{D}(S)$. Moreover, if $B\subseteq \mathcal{D}(S)$, then $\mathrm{T}(B) = \cup_{b\in B} \mathrm{T}(b)$. The following result is a reformulation of Proposition~\ref{prop-19} using this new notation.

\begin{proposition}\label{prop-20}
	If $S\in\mathscr{I}(A,F)$ and $\mathcal{K}(S) = \{ \mathrm{T}(B) \mid B\subseteq \mathcal{D}(S) \}$, then $[S] = \left\{ \Gamma(S)\cup X \mid X\in\mathcal{K}(S) \right\}$.
\end{proposition}

We are ready to show the algorithm for computing $[S]$ when $S\in\mathscr{I}(A,F)$.

\begin{algorithm}\label{alg-21}
	\mbox{ } \par	
	\noindent $\ $ INPUT: $S \in \mathscr{I}(A,F)$. \par \smallskip
	\noindent $\ $ OUTPUT: $[S]$.
	\vspace{-3pt}
	\begin{itemize}
		\setlength\itemsep{0pt}
		\item[(1)] Compute $\Gamma(S)$ and $\mathcal{D}(S)$.
		\item[(2)] Compute $\mathcal{K}(S) = \{ \mathrm{T}(B) \mid B\subseteq \mathcal{D}(S) \}$.
		\item[(3)] Return $[S] = \left\{ \Gamma(S)\cup X \mid X\in\mathcal{K}(S) \right\}$.
	\end{itemize}
\end{algorithm}

\begin{remark}\label{rem-21}
	The results of this section allow us to have an algorithmic process to compute $\mathscr{S}(A,F)$. Specifically, $\mathscr{S}(A,F) = \cup_{S\in\mathscr{I}(A,F)}[S]$. Moreover, since $\mathscr{S}(\{0\},F)$ is the set of numerical semigroups with Frobenius number equal to $F$, this process is a generalisation of the algorithm shown in \cite{computer} for computing all the numerical semigroups with a given Frobenius number.
\end{remark}

Let us illustrate with an example how the above algorithmic process works.

\begin{example}\label{exmp-21}
	Let us compute $\mathscr{S}(\{4\},11)$.
	\begin{itemize}
		\setlength\itemsep{0pt}
		\item From Example~\ref{exmp-16}, we have that $\mathscr{I}(\{4\},11) = \left\{ \langle 2,13 \rangle, \langle 4,5 \rangle, \langle 4,6,9 \rangle \right\}$.
		\item If $S=\langle 2,13 \rangle$, then $\Gamma(S)=S$ and $\mathcal{D}(S)=\emptyset$. Thus, $[\langle 2,13 \rangle] = \{ \langle 2,13 \rangle\}$.
		\item If $S=\langle 4,5 \rangle$, then $\Gamma(S)=S$ and $\mathcal{D}(S)=\emptyset$. Thus, $[\langle 4,5 \rangle] = \{ \langle 4,5 \rangle\}$.
		\item If $S=\langle 4,6,9 \rangle$, then $\Gamma(S)=\langle 4,13,14,15 \rangle$ and $\mathcal{D}(S)=\{6,9,10\}$.
		\begin{itemize}[label=$\ast$]
			\item If $B=\emptyset$, then $\mathrm{T}(B)=\emptyset$.
			\item If $B=\{6\}$, then $\mathrm{T}(B)=\{6,10\}$.
			\item If $B=\{9\}$, then $\mathrm{T}(B)=\{9\}$.
			\item If $B=\{10\}$, then $\mathrm{T}(B)=\{10\}$.
			\item If $B=\{6,9\}$, then $\mathrm{T}(B)=\mathrm{T}(\{6\})\cup \mathrm{T}(\{9\})=\{6,9,10\}$.
			\item If $B=\{6,10\}$, then $\mathrm{T}(B)=\mathrm{T}(\{6\})\cup \mathrm{T}(\{10\})=\{6,10\}$.
			\item If $B=\{9,10\}$, then $\mathrm{T}(B)=\mathrm{T}(\{9\})\cup \mathrm{T}(\{10\})=\{9,10\}$.
			\item If $B=\{6,9,10\}$, then $\mathrm{T}(B)=\mathrm{T}(\{6\})\cup \mathrm{T}(\{9\})\cup \mathrm{T}(\{10\})=\{6,9,10\}$.
		\end{itemize}
		Thus, $\mathcal{K}(S) = \left\{ \emptyset, \{9\}, \{10\}, \{6,10\}, \{9,10\}, \{6,9,10\} \right\}$.
		\begin{itemize}[label=$\ast$]
			\item If $X=\emptyset$, then $\Gamma(S)\cup X=\langle 4,13,14,15 \rangle$.
			\item If $X=\{9\}$, then $\Gamma(S)\cup X=\langle 4,9,14,15 \rangle$.
			\item If $X=\{10\}$, then $\Gamma(S)\cup X=\langle 4,10,13,15 \rangle$.
			\item If $X=\{6,10\}$, then $\Gamma(S)\cup X=\langle 4,6,13,15 \rangle $.
			\item If $X=\{9,10\}$, then $\Gamma(S)\cup X=\langle 4,9,10,15 \rangle $.
			\item If $X=\{6,9,10\}$, then $\Gamma(S)\cup X=\langle 4,6,9 \rangle $.
		\end{itemize}
		Thereby,
		\begin{eqnarray}
			[\langle 4,6,9 \rangle] = \{ \langle 4,6,9 \rangle, \langle 4,6,13,15 \rangle, \langle 4,9,10,15 \rangle, \nonumber \\ \langle 4,9,14,15 \rangle, \langle 4,10,13,15 \rangle, \langle 4,13,14,15 \rangle \}. \nonumber
		\end{eqnarray}
	\end{itemize}
	In conclusion,
	\begin{eqnarray}
		\mathscr{S}(\{4\},11) = \{ \langle 2,13 \rangle, \langle 4,5 \rangle, \langle 4,6,9 \rangle, \langle 4,6,13,15 \rangle, \langle 4,9,10,15 \rangle, \nonumber \\ \langle 4,9,14,15 \rangle, \langle 4,10,13,15 \rangle, \langle 4,13,14,15 \rangle \}. \nonumber
	\end{eqnarray}
\end{example}

\section{The set $\mathscr{M}\boldsymbol{(A,B)}$}\label{sect-M(A,B)}

In this section $A$ is a non-empty subset of $\mathbb{N}$ and $B$ is a non-empty finite subset of $\mathbb{N} \setminus \{0\}$.

Firstly, we see an easy result.

\begin{proposition}\label{prop-22}
	$\mathscr{S}(A,B)\not=\emptyset$ if and only if $B\cap\langle A \rangle=\emptyset$.
\end{proposition}

\begin{proof}
	If $S\in\mathscr{S}(A,B)$, then $\langle A \rangle\subseteq S$ and $B\cap S=\emptyset$. Therefore, $B\cap\langle A \rangle=\emptyset$.
	
	For the other implication, it is enough to note that, if $B\cap\langle A \rangle=\emptyset$, then $\langle A \rangle \cup \{\max(B)+1,\to\} \in \mathscr{S}(A,B)$.
\end{proof}

As a consequence of Proposition~\ref{prop-22} we have the following result.

\begin{corollary}\label{cor-23}
	$\mathscr{M}(A,B)\not=\emptyset$ if and only if $B\cap\langle A \rangle=\emptyset$.
\end{corollary}

In the next result we see how $\mathscr{S}(A,B)$ can be built from the elements of $B$.

\begin{proposition}\label{prop-24}
	If $B=\{b_1,\ldots,b_r\}$, then
	\[ \mathscr{S}(A,B) = \left\{ S_1\cap\ldots\cap S_r \mid S_i\in \mathscr{S}(A,\{b_i\}), \; 1 \leq i \leq r \right\}. \]
\end{proposition}

\begin{proof}
	Firstly, it is clear that, if $S_i\in \mathscr{S}(A,\{b_i\})$ for all $i\in\{1,\ldots,r\}$, then $S_1\cap\ldots\cap S_r \in \mathscr{S}(A,B)$.
	
	Secondly, if $S \in \mathscr{S}(A,B)$, then $S \in \mathscr{S}(A,\{b_i\})$ for all $i\in\{1,\ldots,r\}$. Consequently, by setting $S_1=\cdots=S_r=S$, we have that $S=S_1\cap\ldots\cap S_r$ with $S_i \in \mathscr{S}(A,\{b_i\})$ for all $i\in\{1,\ldots,r\}$.
\end{proof}

The following lemma is clear from the definitions.

\begin{lemma}\label{lem-25}
	If $b$ is a positive integer, $b\notin \langle A \rangle$, and $S\in \mathscr{S}(A,\{b\})$, then there exists $T\in\mathscr{M}(A,\{b\})$ such that $S\subseteq T$.
\end{lemma}

Let us now see how $\mathscr{M}(A,B)$ can be built from the elements of $B$.

\begin{theorem}\label{thm-26}
	If $B=\{b_1,\ldots,b_r\}$, then
	\[ \mathscr{M}(A,B) = \mathrm{Maximal}\big(\left\{ S_1\cap\ldots\cap S_r \mid S_i\in \mathscr{M}(A,\{b_i\}), \; 1 \leq i \leq r \right\}\big). \]
\end{theorem}

\begin{proof}
	If $S\in \mathscr{M}(A,B)$, then $S\in\mathscr{S}(A,B)$ and, by Proposition~\ref{prop-24}, there exists $S_i\in\mathscr{S}(A,\{b_i\})$, $i\in\{1,\ldots,r\}$, such that $S=S_1\cap\ldots\cap S_r$. Moreover, from Lemma~\ref{lem-25}, there exists $T_i\in\mathscr{S}(M,\{b_i\})$, $i\in\{1,\ldots,r\}$, such that $S_i\subseteq T_i$. Thus, if $T=T_1\cap\ldots\cap T_r$, then $S\subseteq T$ and, again by Proposition~\ref{prop-24}, we have that $T\in\mathscr{S}(A,B)$. Thereby, since $S\in \mathscr{M}(A,B)$, we conclude that $S=T$ and, consequently,
	\[ \mathscr{M}(A,B) \subseteq \left\{ S_1\cap\ldots\cap S_r \mid S_i\in \mathscr{M}(A,\{b_i\}), \; 1 \leq i \leq r \right\}. \]
	Now, from Proposition~\ref{prop-24}, we have that
	\[ \left\{ S_1\cap\ldots\cap S_r \mid S_i\in \mathscr{M}(A,\{b_i\}), \; 1 \leq i \leq r \right\} \subseteq \mathscr{S}(A,B). \]
	From here, the conclusion follows immediately.
\end{proof}

The next result is easy to verify.

\begin{lemma}\label{lem-27}
	If $S$ and $T$ are numerical semigroups, then $\mathrm{F}(S\cap T)$ is equal to $\max\{\mathrm{F}(S),\mathrm{F}(T)\}$.
\end{lemma}

In the following result we determine $\mathscr{M}(A,\{b\})$.

\begin{proposition}\label{prop-28}
	If $b$ is a positive integer such that $b\notin \langle A \rangle$, then $\mathscr{M}(A,\{b\}) = \mathscr{I}(A,b)$.
\end{proposition}

\begin{proof}
	If $S\in \mathscr{M}(A,\{b\})$, then we can easily deduce that $A\subseteq S$ and $\mathrm{F}(S)=b$. Let us suppose now that $S$ is not irreducible. In such a case, there exist two numerical semigroups, $T_1$ and $T_2$, such that $S\subsetneq T_1$, $S\subsetneq T_2$, and $S=T_1\cap T_2$. From Lemma~\ref{lem-27}, we know that $\mathrm{F}(S)=\max\{\mathrm{F}(T_1),\mathrm{F}(T_2)\}$. In particular, and without loss of generality, we can assume that $\mathrm{F}(S)=\mathrm{F}(T_1)$. Thus, $T_1\in\mathscr{S}(A,\{b\})$, $S\subsetneq T_1$, and hence $S\notin \mathscr{M}(A,\{b\})$, which is a contradiction.
	
	For the reverse inclusion, let us take $S\in\mathscr{I}(A,b)$ and $T\in \mathscr{S}(A,\{b\})$ such that $S\subseteq T$. Then, by item 3 of Lemma~\ref{lem-2}, we get that $S=T$ and, thereby, $S\in\mathscr{M}(A,\{b\})$.
\end{proof}

As a consequence of Theorem~\ref{thm-26} and Propositions~\ref{prop-30} and \ref{prop-28}, we have the next result. Recall that, if $\mathrm{Ap}(S,n)=\{w(0),w(1),\ldots,w(n-1)\}$, then we denote by $\theta_n(S)=(w(1),\ldots,w(n-1))$.

\begin{theorem}\label{thm-31}
	If $B=\{b_1,\ldots,b_r\}$ and $n=\max(B)+1$, then
	\[ \mathscr{M}(A,B) = \big\{ S\in\mathscr{S}_n \mid \theta_n(S)\in\mathrm{Minimal}(\Theta(A,B)) \big\}  \]
	where $\Theta(A,B) = \left\{ \theta(S_1)\lor\cdots\lor\theta(S_r) \mid S_i\in\mathscr{I}(A,b_i), 1\leq i\leq r \right\}$.
\end{theorem}

If $S$ is a numerical semigroup and $n\in S\setminus\{0\}$, then we can obtain a list $[x_0,\ldots,x_{n-1}]$ such that $\mathrm{Ap}(S,n) = \{w(0)=x_0,\ldots,w(n-1)=x_{n-1}\}$ by using the function \texttt{AperyListOfNumericalSemigroupWRTElement(S,n)} of \cite{numericalsgps}. Thus, we have an algorithmic procedure for computing $\theta_n(S)$ from a minimal system of generators of $S$ and an element $n\in S\setminus\{0\}$.

Let us now show the algorithm for computing $\mathscr{M}(A,B)$.

\begin{algorithm}\label{alg-32}
	\mbox{ } \par	
	\noindent $\ $ INPUT: $A$ (non-empty subset of $\mathbb{N}$) and $B=\{b_1,\ldots,b_r\}$ (non-empty finite subset of $\mathbb{N} \setminus \{0\}$) such that $\langle A \rangle\cap B=\emptyset$. \par \smallskip
	\noindent $\ $ OUTPUT: $\mathscr{M}(A,B)$.
	\vspace{-3pt}
	\begin{itemize}
		\setlength\itemsep{0pt}
		\item[(1)] Using Algorithm~\ref{alg-15}, compute $\mathscr{I}(A,b_i)$ for all $i\in\{1,\ldots,r\}$.
		\item[(2)] $n=\max(B)+1$.
		\item[(3)] For each $i\in\{1,\ldots,r\}$, compute $E_i=\{ \theta_n(S) \mid S\in\mathscr{I}(A,b_i) \}$.
		\item[(4)] $E=\left\{\alpha_1\lor\cdots\lor\alpha_r \mid \alpha_i\in E_i \mbox{ for all } i\in\{1,\ldots,r\} \right\}$.
		\item[(5)] $\mathscr{E} = \mathrm{Minimal}(E)$.
		\item[(6)] Return $\big\{ \langle \{x_1,\ldots,x_{n-1},n\} \rangle \mid (x_1,\ldots,x_{n-1})\in \mathscr{E} \big\}$.
	\end{itemize}
\end{algorithm}

\begin{remark}\label{rem-32}
	In \cite{royal} it is shown an algorithm that allows to compute all the numerical semigroups which are maximal in the set of numerical semigroups that have empty intersection with $B$. Considering that this set is $\mathscr{M}(\{0\},B)$, we have that Algorithm~\ref{alg-32} is a generalisation of the algorithm seen in \cite{royal}.
\end{remark}

Let us see, in an example, how Algorithm~\ref{alg-32} works.

\begin{example}\label{exmp-32}
	Let us compute $\mathscr{M}(\{4,9\},\{11,14\})$.
	\begin{itemize}
		\setlength\itemsep{0pt}
		\item From Algorithm~\ref{alg-15}, we have that
		\begin{itemize}[label=$\ast$]
			\item $\mathscr{I}(\{4,9\},11) = \left\{ \langle 4,15 \rangle, \langle 4,6,9 \rangle \right\}$.
			\item $\mathscr{I}(\{4,9\},14) = \left\{ \langle 4,9,11 \rangle \right\}$.
		\end{itemize}
		\item $n=\max(\{11,14\})+1=15$.
		\item Let $b_1=11$ and $b_2=14$.
		\begin{itemize}[label=$\ast$]
			\item $E_1=\{\alpha_{11},\alpha_{12}\} = \{ (16,17,18,4,20,6,22,8,9,10,26,12,13,14), \\ \mbox{} \hspace{3cm} (16,17,18,4,5,21,22,8,9,10,26,12,13,14) \}$.
			\item $E_2=\{\alpha_2\} = \{ (16,17,18,4,20,21,22,8,9,25,11,12,13,29) \}$.
		\end{itemize}
		\item $E=\left\{\alpha_{11}\lor\alpha_2, \alpha_{12}\lor\alpha_2\right\} = \\ \mbox{} \hspace{0.75cm} \{(16,17,18,4,20,21,22,8,9,25,26,12,13,29)\}$.
		\item $\mathscr{E} = \{(16,17,18,4,20,21,22,8,9,25,26,12,13,29)\}$.
		\item $\mathscr{M}(\{4,9\},\{11,14\}) = \left\{ \langle 15,16,17,18,4,20,21,22,8,9,25,26,12,13,29 \rangle \right\}$.
	\end{itemize}
	In conclusion,
	\[ \mathscr{M}(\{4,9\},\{11,14\}) = \left\{ \langle 4,9,15 \rangle \right\}. \]
\end{example}

\section{The algorithm}\label{sect-problem}

We are now in conditions to show the algorithm for computing $\mathfrak{m}(A,B)$ and, in this way, for giving an answer to Problem~\ref{prob}.

\begin{algorithm}\label{alg-38}
	\mbox{ } \par	
	\noindent $\ $ INPUT: $A$ (non-empty subset of $\mathbb{N}$) and $B$ (non-empty finite subset of $\mathbb{N} \setminus \{0\}$) such that $\langle A \rangle\cap B=\emptyset$. \par \smallskip
	\noindent $\ $ OUTPUT: $\mathfrak{m}(A,B)$.
	\vspace{-3pt}
	\begin{itemize}
		\setlength\itemsep{0pt}
		\item[(1)] Using Algorithm~\ref{alg-32}, compute $\mathscr{M}(A,B)$.
		\item[(2)] Return $ \{ \mathbb{N}\setminus S \mid S \in \mathscr{M}(A,B) \} $.
	\end{itemize}
\end{algorithm}

\begin{remark}\label{rem-38}
	In \cite{royal} it is shown an algorithm that computes all sets $C \subseteq \mathbb{N}\setminus\{0\}$ that are minimal with the condition that, if $(x_1,\ldots,x_n) \in \mathfrak{P}(B)$, then $\{x_1,\ldots,x_n\} \cap C \not= \emptyset$. Let us observe that such sets are nothing but the elements of $\mathfrak{m}(\{0\},B)$. Therefore, Algorithm~\ref{alg-38} is a generalisation of the algorithm seen in \cite{royal}.
\end{remark}

\begin{example}
	From Example~\ref{exmp-32} we deduce that a set $K$ satisfies the conditions
	\begin{itemize}
		\item $K \cap \langle 4,9 \rangle = \emptyset$,
		\item for every partition of 11 or 14, there is at least one summand that belongs to $K$,
	\end{itemize}
	if and only if $K$ contains the set $\mathbb{N} \setminus \langle 4,9,15 \rangle = \{1,2,3,5,6,7,10,11,14\}$.
\end{example}

\section*{Acknowledgement}

Both authors are supported by the project MTM2017-84890-P (funded by Mi\-nis\-terio de Econom\'{\i}a, Industria y Competitividad and Fondo Europeo de Desarrollo Regional FEDER) and by the Junta de Andaluc\'{\i}a Grant Number FQM-343.

\bigskip

\noindent \textbf{Data availability statement.} This manuscript has no associated data.

\end{document}